\documentclass{amsart}[12pt]
\def\doctype{}

\usepackage{latexsym,amssymb}
\usepackage{color}
\usepackage{fancyhdr}
\usepackage{hyperref}

\newcommand\cX{\mathfrak{X}}
\newcommand\cF{\mathcal{F}}
\newcommand\vj{\mathbf{1}}
\newcommand\vb{\mathbf{b}}
\newcommand\ve{\mathbf{e}}
\newcommand\vx{\mathbf{x}}

\newcommand\vo{\mathbf{0}}

\newcommand\Q{\mathbb{Q}}

\newcommand\decomp{\trianglelefteq_\Q}

\newcommand\Sym{\mathcal{S}}
\newcommand\lam{\lambda}

\newcommand{\comment}[1]{}

\numberwithin{equation}{section}

\fancypagestyle{titlepage}{
\fancyhead[R]{\doctype}
\fancyhead[CL]{}
\cfoot{\vspace{5pt} \thepage}
}

\let\oldsection\section
\newcommand\boldsection[1]{\oldsection{\bf #1}}
\newcommand\starsection[1]{\oldsection*{\bf #1}}
\makeatletter
\renewcommand\section{\@ifstar\starsection\boldsection}
\makeatother

\newtheoremstyle{theorem}
  {12pt}		  % space above
  {0pt}  % space below
  {\sl}  % bofy font
  {\parindent}     % ident - empty=no indent,  \parindent= paragraph indent
  {\bf}  % thm head font
  {. }    % punctuation after thm head
  { }    % space after thm head: `` ``=normal \newline=linebreak
  {}     % thm head specification
\theoremstyle{theorem}
\newtheorem{thm}{Theorem}[section]  % 1st argument is your name for it
\newtheorem{lemma}[thm]{Lemma}     % 2nd argument is what is printed
\newtheorem{cor}[thm]{Corollary}

\newtheorem{prop}[thm]{Proposition}

\newtheoremstyle{definition}
  {12pt}		  % space above
  {0pt}  % space below
  {}  % bofy font
  {\parindent}     % ident - empty=no indent,  \parindent= paragraph indent
  {\bf}  % thm head font
  {. }    % punctuation after thm head
  { }    % space after thm head: `` ``=normal \newline=linebreak
  {}     % thm head specification
\theoremstyle{definition}

\newtheorem{ex}[thm]{Example}

\newcommand\rk{{\sc Remark.} }
\newcommand\rks{{\sc Remarks.} }

\renewcommand{\proofname}{Proof}

\makeatletter
\renewenvironment{proof}[1][\proofname]{\par
  \pushQED{\qed}%
  \normalfont \partopsep=\z@skip \topsep=\z@skip
  \trivlist
  \item[\hskip\labelsep
        \scshape
    #1\@addpunct{.}]\ignorespaces
}{%
  \popQED\endtrivlist\@endpefalse
}
\makeatother

\makeatletter
\renewcommand*\@maketitle{%
  \normalfont\normalsize
  \@adminfootnotes
  \@mkboth{\@nx\shortauthors}{\@nx\shorttitle}%
  \global\topskip42\p@\relax % 5.5pc   "   "   "     "     "
  \@settitle
  \ifx\@empty\authors \else {\vskip 1em
\vtop{\centering\shortauthors\@@par}} \fi
  \ifx\@empty\@date \else {\vskip 1em \vtop{\centering\@date\@@par}}\fi % MY CHANGE
  \ifx\@empty\@dedicatory
  \else
    \baselineskip18\p@
    \vtop{\centering{\footnotesize\itshape\@dedicatory\@@par}%
      \global\dimen@i\prevdepth}\prevdepth\dimen@i
  \fi
  \@setabstract
  \normalsize
  \if@titlepage
    \newpage
  \else
    \dimen@34\p@ \advance\dimen@-\baselineskip
    \vskip\dimen@\relax
  \fi
} % end \@maketitle
\renewcommand*\@adminfootnotes{%
  \let\@makefnmark\relax  \let\@thefnmark\relax
  \ifx\@empty\@subjclass\else \@footnotetext{\@setsubjclass}\fi
  \ifx\@empty\@keywords\else \@footnotetext{\@setkeywords}\fi
  \ifx\@empty\thankses\else \@footnotetext{%
    \def\par{\let\par\@par}\@setthanks}%
  \fi
\thispagestyle{titlepage}
}
\makeatother

\begin{document}

\title{\large Fractional edge-decompositions of dense graphs and related eigenvalues}

\author{Peter J.~Dukes}
\address{
Mathematics and Statistics,
University of Victoria, Victoria, Canada
}
\email{dukes@uvic.ca}

\thanks{Research of the first author is supported by NSERC grant number 312595--2010}

\date{\today}

\begin{abstract}
We consider the problem of decomposing some $t$-uniform hypergraph $G$ into copies of another, say $H$, with nonnegative rational weights. For fixed $H$ on $k$ vertices, we show that this is always possible for all $G$ having sufficiently many vertices and `local density' at least $1-C(t)k^{-2t}$.  In the case $t=2$ and $H=K_3$, we show that all large graphs with density at least $27/28$ admit a fractional triangle decomposition.
The proof relies on estimates of certain eigenvalues in the Johnson scheme.  
\end{abstract}

\maketitle
\hrule

%\tableofcontents \hrule

%\setpagewiselinenumbers
%\modulolinenumbers[5]
%\linenumbers

\setcounter{section}{-1}
\section{Status}

We begin with an important clarification.  The author regrets an error in his original paper \cite{ratdec} on this topic.  The main result, Theorem~\ref{main}, stays valid because of its conservative statement.  However, an incorrect step in its proof led to wrong constants for thresholds in the ensuing discussion.  For posterity, most of \cite{ratdec} is mirrored in Sections 1 to 5 which follow, with some housekeeping applied.  A corrigendum fixing the incorrect step appears in Section 6 and has been submitted to the journal.  

A future version of this may integrate the corrections into the original article for coherence.

\section{Introduction}

Let $t$ be a positive integer.  The set of all $t$-element subsets of a
set $X$ is written $\binom{X}{t}$. By a (rational) $t$-{\em vector on}
$X$, we mean a function $f \in \Q^{\binom{X}{t}}$. 

A $t$-{\em uniform hypergraph}, or simply $t$-{\em graph}, is a triple $H=(X,E,\iota)$, 
where 
\begin{itemize}
\item
$X$ is a set of {\em points} or {\em vertices}, 
\item
$E$ is a set of {\em edges}, and 
\item
$\iota \subset X \times E$ is an {\em incidence} relation such that every edge is incident
with precisely $t$ different vertices.
\end{itemize}
Edges are usually identified with the set of incident vertices, dispensing with $\iota$.  However, the definition above permits `multiple edges'.  If there are no multiple edges, then $H$ is said to be {\em simple}.
Unless otherwise specified, all $t$-graphs are assumed simple, and $E \subseteq \binom{X}{t}$.  With this understanding, we may conveniently identify $t$-graphs with $(0,1)$ $t$-vectors.

A $t$-graph $H'$ with vertex set $X'$ and edge set $E'$ is a {\em subgraph} of $H$ if $X' \subseteq X$ and $E' \subseteq E$.  The corresponding $t$-vectors satisfy $f' \le f|_{\binom{X'}{t}}$.

Ordinary graphs are $2$-graphs; note however that the definition does not allow `loops'.

Consider a large $t$-graph $G$ on vertex set $V$, $|V|=v$.  For $0 \le s \le t$, the {\em degree} in $G$ of an $s$-subset $S$ of vertices is the number of edges of $G$ which contain $S$.  The minimum degree over all $s$-subsets is denoted $\delta_s(H)$.  Degrees of $(t-1)$-subsets are normally called \emph{codegrees}.  

We say that $G$ is $(1-\epsilon)$-{\em dense} if $\delta_{t-1}(G) \ge (1-\epsilon)(v-t+1)$.  In other words, a $t$-graph is $(1-\epsilon)$-dense if, given any $t-1$ points, the probability that another point fails to induce an edge is at most $\epsilon$.

The {\em complete} $t$-graph or {\em clique} on $V$ has edges $\binom{V}{t}$ and is equivalent to the constant $t$-vector with every coordinate equal to $1$.  The standard graph-theoretic notation is $K^{t}_v$, where the superscript is normally omitted if $t=2$, or if it is otherwise understood.  Of course, complete $t$-graphs are 1-dense.

Suppose $G$ and $H$ are $t$-graphs, as above, with respective vertex sets $V$ and $X$.  A {\em fractional} or {\em rational decomposition} of $G$ into copies of $H$ is a set of pairs $(H_i,w_i)$, where
\vspace{-11pt}
\begin{itemize}
\item
each $H_i$ is a subgraph of $G$ isomorphic to $H$;
\item
$w_i$ are positive weights such that, for every edge $T$ of $G$,
\begin{equation}
\label{dec-system}
\sum_{i: T \in H_i} w_i = 1.
\end{equation}
\end{itemize}
To be clear, $T \in H_i$ means that $T$ is an edge of $H_i$.

Although the existence questions for fractional decompositions are interesting in their own right, there are actually some nice applications where fractional weights are allowed -- even desired -- such as in statistics (balanced sampling plans) and electrical engineering (network scheduling).

Since (\ref{dec-system}) leads to a linear system with integral coefficients, there is no loss in generality in assuming $w_i \in \Q$.
Note that if the $w_i$ are integers (0 or 1), the result is an ordinary edge-decomposition.  Although we do not need the notation very frequently, a reasonable abbreviation is $H \decomp G$ for rational decomposition and $H \trianglelefteq G$ for ordinary decomposition.

Alternative descriptions are possible.  For instance, if $H$ has vertex set $X$, a fractional decomposition of $G$ into copies of $H$ can be viewed as a nonnegative formal linear combination of injections $X$ into $V$, say $\sigma \in \Q_{\ge 0}[X \hookrightarrow V]$, so that $\sigma H = G$.

A (signed) linear combination of injections $\sigma \in \Q[X \hookrightarrow V]$ is not enough, as the following example shows.

\begin{ex}
Here $t=2$. 
Let $G=C_5$ be the 5-cycle $12345$ on $V=\{1,2,3,4,5\}$, and let $H=K_3$ on a three element set $X$.
Then combining `positive' copies of $H$ on $123,145$ plus a `negative' 
copy of $H$ on $134$ yields a 2-vector $G'$ with pairs $\{1,2\}$, 
$\{2,3\}$, $\{4,5\}$, $\{1,5\}$ having weight 1, pair $\{3,4\}$ having 
weight $-1$, and all other pairs having weight 0.
So the five cyclic shifts of $G'$ combine to yield $3G$ (the 5-cycle with 
every edge tripled).  Therefore, there exists $\sigma \in \Q[X \hookrightarrow V]$ with $\sigma H = 
G$.  However, since $H$ is not a subgraph of $G$, it is clear that there is no such $\sigma  
\in \Q_{\ge 0}[X \hookrightarrow V]$. 
\end{ex}

For ordinary graphs $G$ and $H$, another equivalent formulation arises from the adjacency matrices  $A_G$ and $A_H$.  It is easy to see that $H \decomp G$ (respectively $H \trianglelefteq G$) is equivalent to a decomposition $$A_G = \sum w_i Q_i^\top A_H Q_i,$$
where $Q_i$ are $|X| \times |V|$ $(0,1)$ `injection' matrices having row sum $1$, and $w_i$ are positive rationals (integers).

The following facts are evident from the definitions.

\begin{lemma}
\label{evident}
\begin{enumerate}
\renewcommand{\theenumi}{\alph{enumi}}
\item
Both $\decomp$ and $\trianglelefteq$ are transitive on $t$-graphs.
\item
If $H$ is a $t$-graph with $p \le v$ vertices and $q >0$ edges, then 
$H \decomp K_v^t$.
\end{enumerate}
\end{lemma}

\rk
Part (a) is quite clear.  For (b), it is enough to take each labeled subgraph of $H$ in the complete graph with weight $\binom{v}{t}/q p! \binom{v}{p}$.

Obviously, for $H \decomp G$, it is necessary that $H$ be a subgraph of $G$. In fact, every $t-1$ elements of $G$ must belong to enough copies of $H$ to exhaust the degree at that vertex.    
For instance, large balanced complete bipartite $2$-graphs $G$ are nearly $\frac12$-dense but triangle-free.  Edges can be thrown in until $G$ becomes nearly $\frac34$-dense and still admit no decomposition into copies of $K_3$.  Actually, not much more is known about the density of $G$ failing to admit a decomposition apart from this kind of counting analysis.  The weak (full) {\em Nash-Williams conjecture} states that $K_3 \decomp G$ (resp. $K_3 \trianglelefteq G$) provided that $G$ is at least $\frac{3}{4}$-dense (and, both locally and globally, $K_3$-divisible).

In this paper, we prove the following existence result on fractional decompositions of dense hypergraphs.

\begin{thm}
\label{main}
For integers $k \ge t \ge 2$, there exists $v_0(t,k)$ and $C=C(t)$ such that,
for $v > v_0$ and $\epsilon < Ck^{-2t}$, any $(1-\epsilon)$-dense $t$-graph $G$ on $v$ vertices
admits a fractional decomposition into copies of $K_k$.
\end{thm}

By Lemma~\ref{evident}, the same result holds for any $t$-graph $H$ on $k$ vertices replacing $K_k$.

In \cite{Yuster2}, Yuster proved the same result for $\epsilon \lessapprox 6^{-kt}$, although it was admitted that small improvements may be possible.  Probabilistic and combinatorial arguments were central.  A better result was obtained for ordinary graphs, proved in \cite{Yuster1} for $\epsilon \le 1/9k^{10}$.  

Here, the improvement from Theorem~\ref{main} is substantial, with a qualitative weakening on the density requirement for general $t$, and a bound much closer to the density condition for ordinary graphs.  Our new upper bound on $\epsilon$ is actually $4^{-t-1} \binom{k}{t}^{-2}$, and small improvements may 
\marginpar{\color[rgb]{1,0,0} Important \\ changes\\ from \cite{ratdec}}
be possible from the present proof technique.  For comparison, our result with $k=3$ and $t=2$ shows that graphs $G$ which are at least $\frac{27}{28}$-dense admit a fractional triangle decomposition.  This is getting much closer to the Nash-Williams bound, though substantial work still remains, even in this basic case.

Our proof of Theorem~\ref{main} is constructive and very na\"ive, at least in principle.  For each edge in $G$, consider the `fan' of all $k$-subsets which cover it and induce a clique $K_k^t$ in $G$.  We actually prove that $G$ is a nonnegative rational combination of its fans.  This is clear for complete hypergraphs $K_v^t$, and so we analyze the small perturbation of the resulting linear system obtained by restricting from $K_v^t$ to $G$.  The outline of the argument is presented in more detail in Sections 2 and 3.  The technicalities amount to estimating certain eigenvalues and norms using the theory of association schemes.  These details are covered in Sections 4, 5 and 6. 

\section{Coverage and linear systems}

Let $V$ be a $v$-set, and suppose that $k \ge t$.  A set system $\cF \subseteq \binom{V}{k}$  is said to {\em cover} $T \in \binom{V}{t}$ exactly $\lam$ times if $T \subset K$ for exactly $\lam$ elements $K \in \cF$.  Alternatively, $\cF$ is a $k$-vector and its {\em coverage} is a $t$-vector $\cF^t$ with
$$\cF^t(T) = \sum_{K \supset T} \cF(K).$$
In context, we may suppress the superscript $t$, and instead write $\cF(T)$ for the coverage of $T$ by $\cF$.

Now, let $\cX=\binom{V}{k}$, fix $U \in \binom{V}{t}$, and consider the family $\cF=\cX[U]$ of all $\binom{v-t}{k-t}$ $k$-subsets of $V$ which contain $U$.  
%Write $\cX^t[U]$ for the coverage.  
Then 
$$\cX[U](T) = \binom{v-|T \cup U|}{k-|T \cup U|},$$
since this counts the number of $k$-subsets containing both $T$ and $U$.
Therefore, we may write
$$\cX[U](T) = \xi_{|T \setminus U|},$$
where
$$\xi_i = \binom{v-t-i}{k-t-i} = \frac{v^{k-t-i}}{(k-t-i)!}+o(v^{k-t-i}).$$ 
for $i=0,1,\dots,t$.
This kind of estimation on the orders 
of binomial coefficients occurs frequently in what follows.

Let $n=\binom{v}{t}$ and identify $\Q^n$ with $\Q^{\binom{V}{t}}$. 
Define the $n \times n$ 
matrix $M$ by 
$$M(T,U) = \xi_{|T \setminus U|} = \cX[U](T),$$
for $T,U \in \binom{V}{t}$.  
In fact, $M$ factors as $M=WW^\top$, where $W$ is the well-known inclusion matrix of $t$-subsets versus $k$-subsets.  However, we do not (at least explicitly) use $W$ in what follows.  

Note that $M^\top = M$ and the constant column (row) sum of $M$ is 
\begin{eqnarray}
\label{sumwts} 
\sum_{T} \xi_{|T \setminus U|} = \sum_{i=0}^t \xi_i \binom{v-t}{i} \binom{t}{i} &=&  \binom{k}{t} \binom{v-t}{k-t} \\
\nonumber
&=& \binom{k}{t} \binom{v}{k-t} + o(v^{k-t}). 
\end{eqnarray} 
Observe that (\ref{sumwts}) counts the number of $k$-subsets intersecting
a given $k$-subset in exactly $t$ points, times the number of choices of a second $t$-subset inside of it.

Although we do not make explicit use of the abundant additional symmetry in $M$, it is worth noting that the symmetric group $\mathcal{S}_V$ induces an action on $\cX$ which stabilizes $M$.

At this point, we note that a nonnegative solution $\vx$ to $M \vx = \vj$ induces a rational decomposition $K_k^t \decomp K_v^{t}$.  Simply take each $\cX[U]$ with weight $\vx(U)$, and the total coverage is 
$$\sum_U \vx(U) M(T,U) = (M \vx)(T) = 1$$
on each $t$-set $T$.  Indeed, $\vj$ is an eigenvector of $M$, and so the unique such $\vx$ simply has the reciprocal of (\ref{sumwts}) in each coordinate.

Decomposing a non-complete $t$-graph $G$ is not so easy.  We must restrict our attention to $k$-subsets that cover only those edges present in $G$.

To this end, define $\cX |_G$ as the family of all $k$-subsets which induce a clique in $G$.  In other words, $K \in \cX |_G$ if and only if 
\begin{itemize}
\item
$K \subseteq V$ with $|K|=k$, and 
\item
$T \subset K$ with $|T|=t$ implies $T$ is an edge of $G$.
\end{itemize}
Note that  $\cX |_G$ is nonempty when $G$ is sufficiently dense.  

Now consider $\cX |_G[U]$, the family of all $k$-subsets on $V$ which contain $U$ and also induce a clique in $G$.  Define the $|G| \times |G|$ matrix $\widehat{M}$, with rows and columns indexed by edges of $G$, by 
$$\widehat{M}(T,U) = \cX |_G[U](T).$$

Again, $\widehat{M}$ is symmetric, since its $(T,U)$-entry just counts the number of $k$-subsets containing $T,U$, and no non-edges of $G$.  And, most importantly, a nonnegative solution $\vx$ to 
\begin{equation}
\label{main-system}
\widehat{M} \vx = \vj,
\end{equation} 
if it exists, yields a rational decomposition of $K_k^t \decomp G$.  Just as in the easy case of complete $t$-graphs above, each $\cX |_G[U]$ is taken with multiplicity $\vx(U)$ to obtain coverage $1$ on edges $T$ of $G$.  By construction, the coverage is also zero on non-edges of $G$.

The basic theme of this article may be summarized as follows: for dense $G$, our matrix $\widehat{M}$ is a small perturbation of the principal submatrix $M |_G$ of $M$, restricted to edges of $G$.  This perturbation will be estimated carefully in the next section; however, the relevant lemma in terms of coverages is given here.

\begin{lemma}
\label{mtilde}
Suppose $G$ is a $(1-\epsilon)$-dense simple $t$-graph. 
\vspace{-11pt}
\begin{enumerate}
\renewcommand{\theenumi}{\alph{enumi}}
\item
Given an edge $T$ and $i$ with $0 \le i \le t$, there are at least 
$$\binom{t}{i} \binom{v}{i} \left[ 1- \binom{t+i}{i} \epsilon +o(1)\right]$$ edges $U$ such that $|T \setminus U|=i$ and $T \cup U$ induces a clique in $G$.
\item
If $T$ and $U$ are edges of $G$ with $|T \setminus U|=i$ and such that $T \cup U$ induces a clique in $G$, then there are at least
$$\binom{v-t-i}{k-t-i} \left[ 1- \left( \binom{k}{t} - \binom{t+i}{i} \right) \epsilon  +o(1)\right]$$
$k$-subsets containing $T \cup U$ and inducing a clique in $G$.
\end{enumerate}
\end{lemma}

\begin{proof}
Let $J$ be a set of $j \ge t$ points which induce a clique $K_j^t$ in $G$.  The number of ways to choose a point $x$ in $V \setminus J$ so that $J \cup \{x\}$ also induces a clique is at least
$v-j-\binom{j}{t-1}z$, 
where $z$ is an upper bound on the number of non-edges incident with each $(t-1)$-subset.  With $z=\epsilon(v-t+1)$, and applying induction, the number of ways to extend $T$ to a clique induced by $T \cup U$, of size $t+i$, is at least
$$\frac{1}{i!} \prod_{t \le j < t+i} \left[ v\left(1- \binom{j}{t-1} \epsilon \right) - O(1) \right].$$
Note the $O(1)$ term depends on $t$ and $\epsilon$ but not on $v$.  
We now expand the dominant term of the product and invoke the inequality
$$\prod_j(1-a_j) \ge 1-\sum_j a_j.$$
Using an identity on the resulting sum of binomial coefficients $\binom{j}{t-1}$, one has the number of such extensions at least
$$\frac{v^i}{i!} \left[ 1- \binom{t+i}{i} \epsilon \right] + o(v^i).$$
Finally, in choosing an edge $U$ (not merely an extension of $T$), we are free to pick any $t-i$ points in $T$.
This proves (a).

Similarly, the number of ways to extend a clique on $T \cup U$ to a clique on $k$ points is at least
$$\frac{1}{(k-t-i)!} \prod_{t+i \le j < k} \left[ v \left(1- \binom{j}{t-1} \epsilon \right) - O(1) \right],$$
or, after expansion and identities,
$$\frac{v^{k-t-i}}{(k-t-i)!} \left[ 1- \left( \binom{k}{t} - \binom{t+i}{i} \right) \epsilon  \right] + o(v^{k-t-i}).$$
This proves (b).
\end{proof}

\rks
Lemma~\ref{mtilde}(a) essentially asserts that `most' entries of $\widehat{M}$ are nonzero, while part (b) asserts that those nonzero entries are close to those of $M$.

\section{Proof of the main theorem}

Our proof relies on a couple of easy facts from linear algebra.  Recall that the matrix norm $||\cdot||_{\infty}$ is induced from the same (max) norm on vectors.  We have $||A||_{\infty}$ equal to the maximum absolute row sum of $A$.
We note below that small perturbations in this norm (actually, in any induced norm) do not destroy positive definiteness.

\begin{lemma}
\label{pd}
Suppose $A$ and $\Delta A$ are Hermitian matrices such that every eigenvalue of $A$ is greater than $||\Delta A||_{\infty}$.  Then $A+\Delta A$ is positive definite.
\end{lemma}

\begin{proof}
This follows easily since the spectral radius (i.e. maximum eigenvalue) of $\Delta A$ satisfies
$$\rho(\Delta A) \le ||\Delta A||_{\infty}.$$
\end{proof}

We will momentarily invoke this fact with $A=M|_G$ and $\Delta A = \Delta M := \widehat{M}-M|_G$.

First though, recall Cramer's rule from college linear algebra.  For non-singular $A$, the system $A \vx = \vb$ has a solution given by 
$$x_i = \frac{\det(A_i)}{\det(A)},$$
where $A_i$ denotes the matrix $A$ with its $i$th column substituted for $\vb$.

Taken together, we conclude that the system (\ref{main-system}) has a positive solution $\vx$ provided
the least eigenvalues of both $M$ and $M_1$ exceed $|| \Delta M ||_\infty$.  Note that we may restrict attention to a single $M_1$ due to invariance of $M$ under the action of $\Sym_V$. 

A careful calculation of the eigenvalues of $M$ and $M_1$ is left for the next 2 sections; however, we summarize the important results here.

\begin{thm}
\label{eig-est}
Asymptotically in $v$, the least eigenvalue of $M$ is $$\theta_t = \binom{v-t}{k-t} + o(v^{k-t}),$$ and the least eigenvalue of $M_1$ is at least $\frac{1}{2} \theta_t$.  
\end{thm}

Of course, the same lower bounds on eigenvalues remain true for the principal submatrices restricted to rows and columns of $M$ indexed by edges of $G$.

Now, it remains to estimate the maximum absolute row sum of $\Delta M$.

\begin{prop}
\label{inf-norm-est}
Let $G$ be a $(1-\epsilon)$-dense simple $t$-graph, and define $\Delta M$ as above.  For small $\epsilon$, and asymptotically in $v$, 
\begin{equation}
\label{inf-norm}
|| \Delta M ||_\infty < \binom{v-t}{k-t} \binom{k}{t}^2 \epsilon + o(v^{k-t}).
\end{equation}
\end{prop}

\begin{proof}
Let $a(i)$ and $b(i)$ denote the expressions given in the statement of Lemma~\ref{mtilde}, parts (a) and (b), respectively.  In row $T$ and columns $U$ with $|T \setminus U|=i$, there are at least $a(i)$ entries where $\widehat{M}$ is nonzero due to $T \cup U$ inducing a clique.  
That is, there are at most $\binom{v-t}{i} \binom{t}{i} - a(i)$ such entries which vanish in $\widehat{M}$.

When $T \cup U$ does induce a clique, we have $\widehat{M} (T,U) \ge b(i)$ and $M(T,U)=\xi_i$.  That is, $\Delta M$ is at most of order $\xi_i - b(i)$ in these entries.

Taken together,
\begin{eqnarray*}
|| \Delta M ||_\infty &<&  \sum_{i=0}^t  \left[ \binom{v-t}{i} \binom{t}{i} -a(i) \right] \xi_i + \binom{v-t}{i} \binom{t}{i} (\xi_i-b(i)) \\
&=& \epsilon \sum_{i=0}^t  \binom{v-t}{i} \binom{t}{i} \binom{k}{t} \xi_i + o(v^{k-t}). \\
\end{eqnarray*}
After invoking (\ref{sumwts}), we obtain the desired bound (\ref{inf-norm}).
\end{proof}

By Lemma~\ref{pd}, Theorem~\ref{eig-est} and Proposition~\ref{inf-norm-est}, 
\marginpar{\color[rgb]{1,0,0} The gap in the \\argument is here.\\  See \S 6 for a fix.}
the vector $\widehat{M}^{-1} \vj$ is (asymptotically in $v$) entrywise positive for 
$$\epsilon < C \binom{k}{t}^{-2}.$$  Therefore, we have an induced fractional decomposition of $G$ into copies of $K_k^{t}$.

We should note that there may be a hope of positive solutions to (\ref{main-system}) for some (possibly all) graphs $G$ even if this worst-case bound for $\epsilon$ were exceeded.

Also, it is probably possible to avoid using Cramer's rule and instead analyze the conditioning number $\kappa(M)$.  However, this is not likely to yield any substantially better bounds on $\epsilon$.

It now remains to prove Theorem~\ref{eig-est}, and this is the subject of the next two sections.

\section{The Johnson scheme and eigenvalue estimates for $M$}

For our purposes, a $k$-{\em class association scheme} on a set $\cX$ consists of $k+1$
nonempty symmetric binary relations $R_0,\dots,R_k$ which partition 
$\cX \times \cX$, such that
\begin{itemize}
\item
$R_0$ is the identity relation, and 
\item
for any $x,y \in \cX$ with $(x,y) \in R_h$, the number of $z \in \cX$ such 
that $(x,z) \in R_i$ and $(z,y) \in R_j$ is the {\em structure constant} 
$p^h_{ij}$ depending only on $h,i,j$.  
\end{itemize}
Let $|\cX|=n$. For $i=0,\dots,k$, define the $n \times n$ {\em adjacency 
matrix} $A_i$, indexed by entries of $\cX$, to have $(x,y)$-entry equal to $1$ if 
$(x,y) \in R_i$, and $0$ otherwise.  It is said that $x$ and $y$ are $i$th {\em associates} when 
$(x,y) \in R_i$.
 
By definition of the structure constants, 
$A_i A_j = \sum_h p^h_{ij} A_h.$
In this way, the adjacency matrices span not only a subspace of the $n \times n$ matrices, but a matrix algebra called the 
{\em Bose-Mesner algebra}.

Interestingly, the adjacency matrices are orthogonal idempotents with
respect to entrywise multiplication, and 
$$A_0+\dots+A_k = J,$$
the all ones matrix.  From spectral theory, the Bose-Mesner algebra also has a 
basis of orthogonal idempotents $E_0,\dots,E_k$ with respect to ordinary matrix
multiplication, and such that 
$$E_0+\dots+E_k=I.$$  
A convention is adopted so that $E_0 = \frac{1}{n} J$, which must be one of these idempotents.  

For more on the theory of association schemes, the reader is directed to 
Chapter 30 of \cite{WvL} for a nice introduction or to Chris Godsil's notes
\cite{godsil} for a very comprehensive reference.

The {\em Johnson scheme} $J(t,v)$
has as elements $\binom{V}{t}$, where $S,T \in \binom{V}{t}$ are declared
to be $i$th associates if and only if $|S \cap T|=t-i$. 

The adjacency matrices and (a certain ordering of) the orthogonal idempotents are related via 
\begin{equation}
\label{aerel}
A_i = \sum_{j=0}^t P_{ij} E_j,
\end{equation} 
where $P=[P_{ij}]$ is the {\em first eigenmatrix}.  For $J(t,v)$, its entries 
are given by
\begin{equation}
\label{hahn}
P_{ij} = \sum_{s=0}^i (-1)^{i-s} \binom{t-s}{i-s} \binom{t-j}{s} 
\binom{v-t+s-j}{s}.
\end{equation}
The expression (\ref{hahn}) is a polynomial of degree $2i$ in $j$.  It is a 
relative of the family of {\em Hahn polynomials}.
From (\ref{aerel}), we have
$$M=\sum_{i=0}^t \xi_i A_i = \sum_{j=0}^t \theta_j E_j,$$
where 
\begin{equation}
\label{eigformula}
\theta_j = \sum_{i=0}^t \xi_i P_{ij}.
\end{equation}
Since  the $E_j$ are orthogonal idempotents, it follows that the 
eigenvalues of $M$ are $\theta_j$, having 
multiplicity 
$$m_j = {\rm rank}(E_j) = \binom{v}{j}-\binom{v}{j-1}.$$
Of course, columns of the $E_j$ are eigenvectors for $\theta_j$.

An easy calculation with convolution identities gives the closed form
\begin{equation*}
\theta_0 = \sum_{i=0}^t \xi_i \binom{t}{i} \binom{v-t}{i} = \binom{v-t}{k-t} \binom{k}{t}.
\end{equation*}
This is simply the row sum of $M$, or (\ref{sumwts}).  The other eigenvalues are 
more complicated, but for our purposes an estimate suffices.  

\begin{prop}
\label{eigsM}
The eigenvalues of $M$ are $\theta_j$, each of multiplicity $m_j = \binom{v}{j}-\binom{v}{j-1}$.  
For sufficiently large $v$, the $\theta_j$ are distinct reals of order $v^{k-t}$.
\end{prop}

\begin{proof} Computing directly from (\ref{hahn}) and (\ref{eigformula}),
\begin{eqnarray*}
\theta_j &=& 
\sum_{i=0}^t \xi_{i} P_{ij}  \\
&=& 
\sum_{i=0}^t \binom{v-t-i}{k-t-i} 
\sum_{s=0}^i (-1)^{i-s} \binom{t-s}{i-s} \binom{t-j}{s}
\binom{v-t+s-j}{s}.
\end{eqnarray*}
Now separating the $s=i$ term,
\begin{eqnarray*}
\theta_j 
&=&
\sum_{i=0}^{t-j} \binom{v-t-i}{k-t-i} 
\binom{v-t+i-j}{i} 
\binom{t-j}{i}+o(v^{k-t})\\
&=&
\frac{1}{(k-t)!}\left[ 
\sum_{i=0}^{t-j}
\binom{k-t}{i} \binom{t-j}{i} \right] v^{k-t} + o(v^{k-t})\\
&=&\frac{1}{(k-t)!}
\binom{k-j}{t-j} v^{k-t} + o(v^{k-t}).
\end{eqnarray*}
The leading coefficient is a multiple of $(k-j)^{k-t}$, which is decreasing in $j$ for $0 \le j \le t$.
This proves the $\theta_j$ are distinct as $v \rightarrow \infty$.
\end{proof}

It should be remarked that similar estimates also appear in Section 4 of \cite{SY}, a recent article on quasi-random hypergraphs.  In any case, the proof of Proposition~\ref{eigsM} also establishes the first part of Theorem~\ref{eig-est}.

\begin{cor}
For large $v$, the least eigenvalue of $M$ is $$\theta_t = \binom{v-t}{k-t} + o(v^{k-t}).$$
\end{cor}

\section{Eigenvalue estimates for $M_1$}

Our focus now shifts to $M_1$.  To this end, define 
$$B=\left[\begin{array}{c|ccc}
1 &  \vo \\
\hline
\vj & I  \\
\end{array}
\right],$$
so that $M_1=MB$ is $M$ with first column replaced by the constant vector $\binom{k}{t} \binom{v-t}{k-t} \vj$.

Observe that the eigenvectors of $B$ are precisely those vectors with first coordinate equal to zero.

The column space of each primitive idempotent $E_j$ for the Johnson scheme $J(t,v)$ can be orthogonally decomposed as 
$$\langle \ve^{(j)} \rangle \oplus \langle \ve^{(j)} \rangle^\perp,$$
where $\ve^{(j)}$ is a unit vector parallel to the first column of $E_j$ and its complement $\langle \ve^{(j)} \rangle^\perp$ is $B$-invariant.

Let $V=[\ve^{(0)} \dots \ve^{(t)}]$ and let $V_0$ be the matrix whose columns are a union of orthonormal bases for the $\langle \ve^{(j)} \rangle^\perp$.

\begin{prop}
\label{eigsM1}
Each eigenvalue of $M_1$ is $v^{k-t}(c+o(1))$, $c$ depending only on $k,t$, and the eigenspace indexing, with $\theta_t/2$ as a lower bound.
\end{prop}

\begin{proof}
Let $Q=[V~ V_0]$, an orthogonal matrix.  Then
$$Q^\top MB Q = 
\left[
\begin{array}{cc}
R & O \\
* & D 
\end{array}
\right],$$
where $R=V^\top MBV$, a $(t+1) \times (t+1)$ matrix, and $D$ is the 
$(n-t-1) \times (n-t-1)$ diagonal 
matrix having eigenvalues $\theta_j$, each with multiplicity $m_j-1$.
It follows that the characteristic polynomial of $M_1=MB$ factors as
$$\chi_{MB}(x) = \chi_{R}(x) \prod_{j=1}^t (x-\theta_j)^{m_j-1}.$$
We recover the original eigenvalues $\theta_j$ as all but $t+1$ of
the eigenvalues of $M_1$.  In light of Proposition~\ref{eigsM}, it remains 
to consider the eigenvalues of $R$.

Let $\Theta = {\rm diag}(\theta_0,\theta_1,\dots,\theta_t)$.
By definition of $V$, we have $MV = V \Theta$.
So, since $M$ is symmetric,
$$R=V^\top MBV = (MV)^\top BV = 
\Theta V^\top BV.$$
It is a routine calculation that 
\begin{equation}
\label{vtbv}
V^\top BV = I+
\left(
\left[
\begin{array}{c}
1 \\
0 \\
\vdots \\
0 
\end{array}
\right]
- \binom{v}{t}^{-1}
\left[
\begin{array}{c}
1 \\
1 \\
\vdots \\
1 
\end{array}
\right]
\right)
\left[
\begin{array}{cccc}
m_0 & m_1 & \dots & m_t
\end{array}
\right].
\end{equation}
The last term on the right of (\ref{vtbv}) is rank one.  Put 
$${\bf u} = \left[
\begin{array}{c}
1 \\
0 \\
\vdots \\
0 
\end{array}
\right]
- \binom{v}{t}^{-1}
\left[
\begin{array}{c}
1 \\
1 \\
\vdots \\
1 
\end{array}
\right]~~~\text{and}~~~{\bf m} = \left[
\begin{array}{c}
m_0 \\
m_1 \\
\vdots \\
m_t 
\end{array}
\right].
$$
Recall for column vectors ${\bf u}$ and ${\bf m}$ the identity
$$\det(I+{\bf u} {\bf m}^\top) = 1+{\bf u}^\top {\bf m}.$$
It follows that the characteristic polynomial of $R$ can be computed rather easily.  
We have
\begin{eqnarray}
\nonumber
\chi_R(x) &=& \det(\Theta(I+ {\bf u} {\bf m}^\top) -xI) \\
\nonumber
&=& (1+  {\bf u}^\top (\Theta-xI)^{-1} \Theta {\bf m})\chi_\Theta(x) \\
\label{chi-r}
&=& \left[1+\frac{\theta_0 m_0}{\theta_0-x} - n^{-1} \sum_{j=0}^t \frac{\theta_j m_j}{\theta_j-x} \right] \chi_\Theta(x).
\end{eqnarray}
Although we are not able to explicitly compute the eigenvalues of $R$ in terms of those of $\Theta$, it is sufficient for our purposes to analyze sign changes and obtain an interlacing result.  For this purpose, consider the rational function $\psi(x) = \chi_R(x)/\chi_\Theta(x)$.  This is the first factor on the right of (\ref{chi-r}).  

Near $\theta_j$, $j>0$, the dominant term in $\psi$ is 
$-n^{-1} \theta_j m_j/(\theta_j-x)$, which changes from negative to positive as $x$ increases.  The opposite is true near $\theta_0$.  

Recall that $\theta_t < \cdots < \theta_1 < \theta_0$, dictating the sign changes of $\chi_\Theta$.
Finally, observe 
\begin{eqnarray*}
\psi ( \theta_t/2 ) &>& 1+1-n^{-1} \sum \frac{\theta_j 
m_j}{\theta_j - \theta_j/2} \\
&=& 2-2(m_0+m_1+\dots + m_t)/n ~=~ 0.
\end{eqnarray*}
These various observations are summarized in Table~\ref{signs}.  It follows that $R$ has $t+1$ different real eigenvalues, each exceeding 
$\frac{1}{2}\theta_t$.  The result now follows from 
Proposition~\ref{eigsM}.
\end{proof}

\begin{table}[htpb]
$$
\begin{array}{c}
{\rm odd}~t~~({\rm even~degree}) \\
\begin{array}{|c|cccccccc|}
\hline
x & \theta_t/2 & \theta_t & \theta_{t-1} & \cdots & \theta_2 & \theta_1 & 
\theta_0 & \infty \\
\hline
\psi(x)&        + & -+ & -+ & \cdots & -+ & -+ & +- & + \\
\chi_\Theta(x)& + & +- & -+ & \cdots & -+ & +- & -+ & + \\
\chi_R(x) &     + & -  & +  & \cdots & + & - & - & +  \\
\hline
\end{array}
\\
~ \\
{\rm even}~t~~({\rm odd~degree}) \\
\begin{array}{|c|cccccccc|}
\hline
x & \theta_t/2 & \theta_t & \theta_{t-1} & \cdots & \theta_2 & 
\theta_1 & \theta_0 & \infty \\
\hline
\psi(x)&        + & -+ & -+ & \cdots & -+ & -+ & +- & + \\
\chi_\Theta(x)& + & +- & -+ & \cdots & +- & -+ & +- & - \\
\chi_R(x) &     + & -  & +  & \cdots & -  & +  & +  & - \\
\hline
\end{array}
\end{array}
$$
\caption{sign changes near eigenvalues of $M$}
\label{signs}
\end{table}

\hrule

\section{Repairing the argument and constants}

The key problem is that Lemma~\ref{pd} (which upper bounds the spectral radius of a perturbation in terms of its $\infty$-norm) does not apply to the non-Hermitian matrix $M_1$. (It may still be the case that $M_1$ satisfies the conclusion.)  

Here we patch the argument in a mostly self-contained format.

As before, let $M$ be the square matrix indexed by $t$-subsets of a $v$-set, and whose $(T,U)$-entry holds the number of cliques $K_k^t$ covering edges $T$ and $U$ in $K_v^t$.
We have 
$$M=\sum_{i=0}^t \xi_i A_i,~~\text{where } \xi_i =  \binom{v-t-i}{k-t-i},$$
lying in the Bose-Mesner algebra of the Johnson scheme $J(t, v)$ generated by adjacency matrices
$\{I = A_0,A_1,\dots,A_t\}$; see \cite{WvL}.
Using a spectral decomposition, we also have $M=\sum_{j=0}^t \theta_j E_j$, with the eigenvalues $\theta_j$ computed as in Section 4.  The largest of these is the rowsum of $M$, or 
$$\theta_0 = \binom{k}{t} \binom{v-t}{k-t}.$$
Let $G$ be a $t$-graph on $v$ vertices which has minimum codegree $\delta_{t-1}(G) \ge \epsilon(v-t+1)$.   Let $\widehat{M}$ have its rows and columns indexed by edges of $G$, and record in its $(T,U)$-entry the number of cliques $K_k^t$ containing both $T$ and $U$.  This is a restricted analog of $M$, now for $G$ instead of $K_v^t$.  We then have a fractional decomposition of $G$ into cliques $K_k^t$ (in fact into `fans' $\cF_T=\{K \subset G: T \in K \cong K_k^t \}$) if  
\begin{equation}
\label{eqn}
\widehat{M} \vx = \theta_0 \vj
\end{equation}
has an entrywise nonnegative (rational) solution $\vx$.   The constant $\theta_0$ appears for convenience only.  As mentioned above, we have $M \vj = \theta_0 \vj$.  
%We wish to restrict and perturb this system to (\ref{eqn}).

To patch the argument for a nonnegative solution of (\ref{eqn}), we follow a strategy in the doctoral dissertation of Kseniya Garaschuk \cite{Kseniya}. The key idea is to use the following error estimate, where recall $|| \cdot ||_\infty$ represents the maximum absolute row sum (entry) for a matrix (vector).

\begin{lemma}
\label{cond}
Let $A \vx = \vb$ be a square system and suppose $A+\Delta A$ is a perturbation with $||A^{-1} \Delta A ||_\infty <1$.  Then $A+\Delta A$ is nonsingular and the unique solution
$\vx + \Delta\vx$ to the equation $(A+\Delta A)(\vx + \Delta \vx) = \vb$ has
\begin{equation}
\label{est}
\frac{||\Delta \vx||_\infty}{||\vx||_\infty} \le \frac{||\Delta A||_\infty ||A^{-1}||_\infty}{1-||A^{-1} \Delta A||_\infty}.
\end{equation}
\end{lemma}
The above is a special case of \cite[\S 37.5, Fact 7]{HLA}. It can be proved using a series expansion of
$(I+A^{-1}\Delta A)^{-1}$ and the triangle inequality.

We apply Lemma~\ref{cond} with $A=M$, $\vx = \vj$, $\vb = \theta_0 \vj$, and
$$A+\Delta A= 
\left[ \begin{array}{c|c}
\widehat{M} & 0 \\ 
\hline
M|_{\overline{G} \times G} & \binom{v-t}{k-t} I  
\end{array}
\right],
$$
where the division of entries (left/right and top/bottom) correspond to edges and non-edges of $G$.  In the lower left, we retain the corresponding entries of $M$.
In \cite{ratdec}, there is a bound on the perturbation $M|_{G \times G} - \widehat{M}$; it is not hard to extend this to a bound on $\Delta A$.  In what follows, we put $n_i=\binom{v-t}{i} \binom{t}{i}$, the number of $t$-sets intersecting a given $t$-set in $t-i$ points, or alternatively the rowsum of $A_i$.

\begin{prop}
\label{norm-perturb}
With $\Delta A$ defined as above,
$$||\Delta A||_\infty < \epsilon \binom{k}{t}^2 \binom{v}{k-t} + o(v^{k-t}).$$
\end{prop}

\begin{proof}
Let $T$ be a fixed edge of $G$, and suppose $i$ is an integer with $0 \le i \le t$.
Following Lemma 2.1(a) of \cite{ratdec}, there are at most $\epsilon \binom{t+i}{i} n_i+o(v^i)$ subsets $U$ with $|U|=t$, $|T \setminus U|=i$, and such that $T \cup U$ fails to induce a clique in $G$.  (The argument is the same whether $U$ is an edge or non-edge of $G$.)  For such $U$, the $(T,U)$-entry of $A+\Delta A$ vanishes, so that $\Delta A(T,U)=-\xi_i$.  Otherwise, when $T \cup U$ does induce a clique, Lemma 2.1(b) of \cite{ratdec} shows that at most $\epsilon (\binom{k}{t}-\binom{t+i}{i}) \xi_i+ o(v^{k-t})$ cliques $K_k^t$ in $K_v^t$ are `broken' in $G$. In total, the row of $\Delta A$ indexed by $T$ has norm at most
$$\epsilon \sum_{i=0}^t [\tbinom{t+i}{i} n_i \xi_i +  n_i (\tbinom{k}{t}-\tbinom{t+i}{i}) \xi_i]+ o(v^{k-t}) =  \epsilon \binom{k}{t} \sum_{i=0}^t n_i\xi_i+ o(v^{k-t}) = \epsilon \binom{k}{t}^2 \binom{v-t}{k-t}+ o(v^{k-t}).$$
Next, by design, the lower left entries of $\Delta A$ are all zero, as are the lower diagonal entries of $\Delta A$.  It follows that rows indexed by non-edges of $G$ have the same upper bound; in fact, we need only use the `first half' of the above estimate.
\end{proof}

%since, incident with each fixed edge of $G$, there are at most $\epsilon \binom{k}{t} \binom{v}{k-t}+o(v^{k-t})$ cliques $K_k^t$ in $K_v^t$ broken by the %restriction to $G$, and each such clique impacts $\binom{k}{t}$ columns of $\Delta A$.  
Next, we provide a bound on $A^{-1}$ coming from calculations in the Bose-Mesner algebra.

\begin{prop}
\label{norm-inverse}
$$||A^{-1}||_\infty <  3^t \left[ \binom{v}{k-t}+o(v^{k-t}) \right]^{-1}.$$
\end{prop}

\begin{proof}
We compute
\begin{align}
\nonumber
||A^{-1}||_\infty & = \left| \left| \sum_{j=0}^t \theta_j^{-1} E_j \right| \right|_\infty & E_j\text{ are orthogonal idempotents}\\
\nonumber
%& = \frac{1}{\binom{v}{t}} \left| \left| \sum_{j=0}^t \theta_j^{-1} \sum_{i=0}^t Q_{ji} A_i  \right| \right|_\infty & \text{Bose-Mesner change of basis; see 
%\cite{WvL}}\\
%\nonumber
&  = \frac{1}{\binom{v}{t}} \left| \left| \sum_{j=0}^t \theta_j^{-1} \sum_{i=0}^t \frac{m_j}{n_i} P_{ij} A_i  \right| \right|_\infty & \text{change of basis; identity (30.5) in \cite{WvL}}\\
\nonumber
&= \frac{1}{\binom{v}{t}} \sum_{i=0}^t ||A_i||_\infty \left| \sum_{j=0}^t \theta_j^{-1} \frac{m_j}{n_i} P_{ij} \right|  & 
\text{swap sums; recall }A_i\text{ are }\{0,1\}\text{-matrices}\\
\nonumber
&= \frac{1}{\binom{v}{t}} \sum_{i=0}^t \sum_{j=0}^t  \left|  \theta_j^{-1} m_j P_{ij} \right|  & 
\text{triangle ineq.; }||A_i||_\infty =  n_i\\
\label{intermediate}
&\le \frac{1}{\binom{v}{t}} \sum_{j=0}^t \theta_j^{-1} m_j \sum_{i=0}^t |P_{ij}|, & \text{swap sums again; } \theta_j, m_j>0
\end{align}
where $m_j:=\binom{v}{j}-\binom{v}{j-1}$ and $P_{ij}:=\sum_{h=0}^i (-1)^{i-h} \binom{t-h}{i-h} \binom{t-j}{h} \binom{v-t+h-j}{h}$ as in \cite{ratdec}.  Note $P_{ij}$ is a polynomial of degree $\min\{i,t-j\}$ in $v$.  Considering dominant terms only, the inner sum in (\ref{intermediate}) is then estimated as 
$$\sum_{i=0}^t |P_{ij}| < \sum_{i=t-j}^{t} \binom{j}{i-j+t} \binom{v}{t-j} +o(v^{t-j}) \le 2^j \binom{v}{t-j} + o(v^{t-j}).$$
Substituting this (and $m_j,\theta_j$) into  (\ref{intermediate}), we have the estimate
\begin{align}
\label{intermediate2}
||A^{-1}||_\infty &< \frac{1}{\binom{v}{t}} \left[ \binom{v}{k-t}+o(v^{k-t})\right]^{-1} \left[\sum_{j=0}^t \binom{k-j}{t-j}^{-1}  2^j \binom{v}{j} \binom{v}{t-j}+o(v^t) \right]\\
\nonumber
& <   \left[ \binom{v}{k-t}+o(v^{k-t})\right]^{-1} \left[\sum_{j=0}^t  2^j \binom{t}{j} +o(1) \right]\\
\nonumber
& < 3^t \left[ \binom{v}{k-t}+o(v^{k-t}) \right]^{-1}.\qedhere
\end{align}
\end{proof}

In our use of Lemma~\ref{cond}, it is sufficient to have $|| \Delta A||_\infty ||A^{-1}||_\infty < \frac{1}{2}$, since the norm is sub-multiplicative and the right side of (\ref{est}) becomes less than one.  It follows from Propositions~\ref{norm-perturb} and \ref{norm-inverse} that we get a positive solution to (\ref{eqn}) for 
$$\epsilon < \frac{1}{2} \cdot 3^{-t} \binom{k}{t}^{-2}$$
and all sufficiently large $v$.  The discussion following Theorem 1.3 in \cite{ratdec} is missing the exponential.

Note we have given away a lot in Proposition~\ref{norm-inverse}  for a clean-looking threshold.  (Note from (\ref{intermediate2}) we use simply $\binom{k-j}{t-j}^{-1} \le 1$, so for large $k$ only the $j=t$ term of the sum is significant.)   
In the case of graphs $(t=2)$, we can compute more carefully, working from (\ref{intermediate2}), to get
$$||A^{-1}||_\infty = \left( 4-\frac{4k-2}{k(k-1)} \right) \left[ \binom{v}{k-2} + o(v^{k-2}) \right]^{-1}.$$
For triangles $(k=3)$ and large $v$, this is about $\frac{7}{3v}$. Together with $||\Delta A||_\infty < 6 \epsilon v$, this specializes to give fractional decompositions of large graphs with $\epsilon < 1/28$.  In fact, this can be further improved to $\epsilon < 1/23$ by estimating $||A^{-1} \Delta A||_\infty$ without using sub-multiplicativity; see \cite{Kseniya} for details.  

Section 5 of \cite{ratdec} (which estimates eigenvalues for the non-Hermitian $M_1$) is not needed, although it is possibly still of some interest.

The author apologizes to those inconvenienced by the error and fines himself \$100 (going to charity) for `speeding' through the argument in \cite{ratdec}.  Thanks to the authors of \cite{Deryk} for encouraging this correction and looking it over.

\end{document}